\documentclass[11pt]{amsart}
\usepackage{amssymb}
\theoremstyle{plain}
\newtheorem{theorem}{Theorem}

\theoremstyle{example}
\newtheorem{example}{Example}

\theoremstyle{definition}

\theoremstyle{remark}

\numberwithin{equation}{section}

\setbox0=\hbox{$+$}
\newdimen\plusheight
\plusheight=\ht0
\def\+{\;\lower\plusheight\hbox{$+$}\;}

\setbox0=\hbox{$-$}
\newdimen\minusheight
\minusheight=\ht0
\def\-{\;\lower\minusheight\hbox{$-$}\;}

\setbox0=\hbox{$\cdots$}
\newdimen\cdotsheight
\cdotsheight=\plusheight%\ht0
\def\cds{\lower\cdotsheight\hbox{$\cdots$}}

\def\cchat{\hat{\mathbb{C}}}

\begin{document}
\title[Divergence of  $q$-Continued Fraction outside
the Unit Circle]
       {The Convergence and  Divergence of  $q$-Continued Fractions
outside
the Unit Circle}
\author{Douglas Bowman}
\address{Department Of Mathematical Sciences,
Northern Illinois University,
 De Kalb, IL 60115}
\email{bowman@math.niu.edu }
\author{James Mc Laughlin}
\address{Mathematics Department,
       Trinity College,
       300 Summit Street,
       Hartford, CT 06106-3100}
\email{james.mclaughlin@trincoll.edu}
\thanks{The second author's research supported in part by a
Trjitzinsky Fellowship.}
\date{May, 11, 2002}
\begin{abstract}

We  consider two classes of $q$-continued fraction whose odd and even
parts are  limit 1-periodic for $|q|>1$,  and give
theorems which guarantee the convergence of the continued fraction,
 or of its odd- and even parts, at points outside the unit circle.

\end{abstract}

\maketitle

\section{Introduction}
Studying the convergence behaviour
of the odd and even parts of  continued fractions is interesting
for a number of different reasons (see, for example, \emph{Section 9.4} of \cite{JT80}).
In this present paper, we examine the
convergence behaviour of $q$-continued fractions outside the unit circle.

Many well-known $q$-continued fractions have the property that their
odd and even parts converge everywhere outside the unit circle.
These include the Rogers-Ramanujan continued fraction,
%{\allowdisplaybreaks
\begin{align*}
%&\phantom{as}\\
K(q) := 1 + \frac{q}{1}
\+
 \frac{q^{2}}{1}
\+
 \frac{q^{3}}{1}
\+
 \frac{q^{4}}{1}
\+
\cds
\end{align*}
%}
and the three
Ramanujan-Selberg continued fractions studied by Zhang in \cite{Z91}, namely,
{\allowdisplaybreaks
\begin{align*}
S_{1}(q):= 1 + \frac{q}{1}
\+
 \frac{q+q^{2}}{1}
\+
 \frac{q^{3}}{1}
\+
 \frac{q^{2}+q^{4}}{1}
\+
\cds ,
\end{align*}
}
%{\allowdisplaybreaks
\begin{align*}
S_{2}(q):=
1 +
 \frac{q+q^{2}}{1}
\+
 \frac{q^{4}}{1}
\+
 \frac{q^{3}+q^{6}}{1}
\+
 \frac{q^{8}}{1}
\+
\cds ,
\end{align*}
%}
and
%{\allowdisplaybreaks
\begin{align*}
S_{3}(q):=
1 +
 \frac{q+q^{2}}{1}
\+
 \frac{q^{2}+q^{4}}{1}
\+
 \frac{q^{3}+q^{6}}{1}
\+
 \frac{q^{4}+q^{8}}{1}
\+
\cds .
\end{align*}
%}
It  was proved in \cite{ABJL92} that if $0<|x| < 1$ then
the odd approximants of $1/K(1/x)$ tend to
%{\allowdisplaybreaks
\begin{align*}
1 - \frac{x}{1}
\+
 \frac{x^2}{1}
\-
 \frac{x^3}{1}
\+\,\cds
\end{align*}
%}
\begin{flushleft}
while the even approximants tend to
\end{flushleft}
%{\allowdisplaybreaks
\begin{align*}
\frac{x}{1}
\+
\frac{x^4}{1}
\+
\frac{x^8}{1}
\+
\frac{x^{12}}{1}
\+\,\cds
.\\
\phantom{as} \notag
\end{align*}
%}
This result was first stated, without proof, by Ramanujan. In \cite{Z91},
Zhang expressed the odd and even parts of each of $S_{1}(q)$,
$S_{2}(q)$ and $S_{3}(q)$  as infinite products,
for $q$ outside the unit circle.

Other $q$-continued fractions have the property that they converge
everywhere outside the unit circle. The most famous example of this latter
type is  G\"{o}llnitz-Gordon continued fraction,
{\allowdisplaybreaks
\begin{align*}
GG(q):=
1+q +
 \frac{q^{2}}{1+q^{3}}
\+
 \frac{q^{4}}{1+q^{5}}
\+
 \frac{q^{6}}{1+q^{7}}
\+\,\cds.
\end{align*}
}

In this present paper we study the convergence behaviour outside the unit circle
 of two families of  $q$-continued fractions, families which include
all of the above continued fractions.

\section{Convergence of the odd and even parts of $q$-continued fractions
 outside the unit circle}\label{sec3}

\vspace{1pt}

 Before coming to our theorems, we need some notation and some
results on limit 1-periodic continued fractions.

Let the $n$-th approximant of the continued fraction
$b_{0}+K_{n=1}^{\infty}a_{n}/b_{n}$ be $P_{n}/Q_{n}$.
The \emph{even} part of $b_{0}+K_{n=1}^{\infty}a_{n}/b_{n}$
 is the continued fraction whose
$n$-th numerator (denominator) convergent equals $P_{2n}$ ($Q_{2n}$),
for $n \geq 0$.
The \emph{odd} part of $b_{0}+K_{n=1}^{\infty}a_{n}/b_{n}$
is the continued fraction whose
zero-th numerator   convergent is $P_{1}/Q_{1}$, whose
zero-th denominator convergent is
$1$, and whose
$n$-th numerator (respectively denominator) convergent equals $P_{2n+1}$
(respectively $Q_{2n+1}$),  for $n \geq 1$.

 For later use we give
explicit expressions for  the odd- and even parts of a continued fraction.
From \cite{LW92}, page 83,
the even part of $b_{0}+K_{n=1}^{\infty}a_{n}/b_{n}$ is
given by
{\allowdisplaybreaks
\begin{multline}\label{ep}
%&\phantom{a} \notag\\
b_{0} +
 \frac{b_{2}a_{1}}{b_{2}b_{1}+a_{2}}
\-
 \frac{a_{2}a_{3}b_{4}/b_{2}}{a_{4}+b_{3}b_{4}+a_{3}b_{4}/b_{2}}
\-
 \frac{a_{4}a_{5}b_{6}/b_{4}}{a_{6}+b_{5}b_{6}+a_{5}b_{6}/b_{4}}
\-
\cds .
%\phantom{a}
\end{multline}
}
From \cite{LW92}, page 85,
the odd part of $b_{0}+K_{n=1}^{\infty}a_{n}/b_{n}$ is
given by
{\allowdisplaybreaks
\begin{multline}\label{op}
%&\phantom{a} \notag\\
\frac{b_{0}b_{1}+a_{1}}{b_{1}} -
 \frac{a_{1}a_{2}b_{3}/b_{1}}{b_{1}(a_{3}+b_{2}b_{3})+a_{2}b_{3}}
\-
 \frac{a_{3}a_{4}b_{5}b_{1}/b_{3}}{a_{5}+b_{4}b_{5}+a_{4}b_{5}/b_{3}}\\
\-
 \frac{a_{5}a_{6}b_{7}/b_{5}}{a_{7}+b_{6}b_{7}+a_{6}b_{7}/b_{5}}
\-
 \frac{a_{7}a_{8}b_{9}/b_{7}}{a_{9}+b_{8}b_{9}+a_{8}b_{9}/b_{7}}
\-
\cds .
%\phantom{a}
\end{multline}
}

\textbf{Definition:}. Let $t(w) = c/(1+w)$, where $c \not = 0$. Let
$x$ and $y$ denote the fixed points of the linear fractional transformation
$t(w)$. Then $t(w)$ is called
{\allowdisplaybreaks
\begin{align}\label{parelox}
&(i)\,\, \text{parabolic,}& &\text{if } x=y,\\
&(ii)\, \,\text{elliptic,}& &\text{if } x \not = y \text{ and }
|1+x|=|1+y|,\notag\\
&(iii)\,\, \text{loxodromic,}&  &\text{if } x \not = y
\text{ and } |1+x| \not = |1+y|.\notag
\end{align}
}
In case (iii), if $|1+x| > |1+y|$, then $\lim_{n \to \infty}t^{n}(w) = x$ for all
$w \not = y$, $x$ is called the \emph{attractive} fixed point
of $t(w)$ and $y$ is called the \emph{repulsive} fixed point
of $t(w)$.

Remark: The above definitions are usually given for more general
linear fractional transformations but we do not need this full
generality here.

The fixed points of $t(w)=c/(1+w)$ are
$x=(-1+\sqrt{1+4c})/2$ and $y=(-1+\sqrt{1+4c})/2$.
It is easy to see that $t(w)$ is parabolic only in the case
 $c=-1/4$, that it is elliptic only when $c$ is a real number
in the interval $(- \infty, -1/4)$ and  that it is loxodromic for
all other values of $c$.

Let $\cchat$  denote the extended complex plane. From \cite{LW92},
pp. 150--151, one has the following theorem.
\begin{theorem}\label{lox}
Suppose $1+ K_{n=1}^{\infty}a_{n}/1$ is limit
$1$-periodic, with $\lim_{n \to \infty}a_{n} = c \not = 0$.
If $t(w) = c/(1+w)$ is loxodromic, then
$1+ K_{n=1}^{\infty}a_{n}/1$ converges to a value $f \in \hat{ \mathbb{C}}$.
\end{theorem}
Remark: In the cases where $t(w)$ is parabolic or elliptic,
whether $1+ K_{n=1}^{\infty}a_{n}/1$ converges or diverges depends
on how the $a_{n}$ converge to c.

We also make use of Worpitzky's Theorem (see \cite{LW92}, pp.
35--36).
\begin{theorem}(Worpitzky) Let the continued fraction
$K_{n=1}^{\infty}a_{n}/1$ be such that $|a_{n}|\leq 1/4$ for $n
\geq 1$. Then $K_{n=1}^{\infty}a_{n}/1$ converges. All
approximants of the continued fraction lie in the disc $|w|<1/2$
and the value of the continued fraction is in the disk
$|w|\leq1/2$.
\end{theorem}

We first consider continued fractions  of the form
{\allowdisplaybreaks
\begin{align*}
G(q):&=1 +
K_{n=1}^{\infty}\frac{a_{n}(q)}{1}:=1+
\frac{f_{1}(q^{0})}{1}
\+\cds \+
\frac{f_{k}(q^{0})}{1}\\
&\+
 \frac{f_{1}(q^{1})}{1}
\+\cds \+
\frac{f_{k}(q^{1})}{1}
\+\cds \+
\frac{f_{1}(q^{n})}{1}
\+ \cds\+
\frac{f_{k}(q^{n})}{1} \+\cds , \notag
\end{align*}
}
where $f_{s}(x)\in \mathbb{Z}[q][x]$, for $1 \leq s
\leq k$. Thus, for $n\geq 0$ and $1 \leq s \leq k$,
{\allowdisplaybreaks
\begin{align}\label{con4}
&a_{nk+s}(q)=f_{s}(q^{n}).&
\end{align}
}
Many well-known $q$-continued fractions, including the
Rogers-Ramanujan continued fraction  and the three Ramanujan-Selberg
continued fractions are of this form, with $k$ at most 2.
Following the example of these four continued fractions,
we make the additional assumptions that,
 for $i \geq 1$,
%{\allowdisplaybreaks
\begin{align}\label{con2}
%&\phantom{as} \notag \\
 \text{degree}(a_{i+1}(q))= \text{degree}(a_{i}(q))+ m, \\
&\phantom{as} \notag
\end{align}
%}
where $m$ is a fixed positive integer, and that
 all of the polynomials $a_{n}(q)$ have the same
leading coefficient. We
prove the following theorem.
\begin{theorem}\label{T4}
 Suppose $G(q) = 1 + K_{n=1}^{\infty}a_{n}(q)/1$
is such that the $a_{n}:= a_{n}(q)$ satisfy \eqref{con4} and
 \eqref{con2}. Suppose further that each $a_{n}(q)$ has the same
 leading coefficient. If  $|q|>1$ then the odd and
even parts of $G(q)$ both converge.
\end{theorem}
Remark: Worpitzky's Theorem gives only that odd- and even parts of
$G(q)$ converge for those $q$ satisfying
$|(1+q^{m})(1+q^{-m})|>4$, a clearly weaker result.
\begin{proof}
Let $|q|>1$.
For ease of notation we write $a_{n}$ for $a_{n}(q)$.
By \eqref{ep}, the even part of $G(q)$ is given by
{\allowdisplaybreaks
\begin{align*}
%&\phantom{a}\notag\\
&G_{e}(q)
:=1 +
 \frac{a_{1}}{1+a_{2}}
\-
 \frac{a_{2}a_{3}}{a_{4}+a_{3}+1}
\-
 \frac{a_{4}a_{5}}{a_{6}+a_{5}+1}
\-
\cds
 \notag\\
%&\phantom{a}\notag\\
& \approx 1 +
 \frac{\displaystyle{
\frac{a_{1}}{1+a_{2}}}}{1}
\-
 \frac{\displaystyle{
\frac{a_{2}a_{3}}{(1+a_{2})(a_{4}+a_{3}+1)}}}{1}
\-
 \frac{\displaystyle{
\frac{a_{4}a_{5}}{(a_{4}+a_{3}+1)(a_{6}+a_{5}+1)}}}{1}
\-
\cds
\notag\\
%&\phantom{a}\notag\\
&= 1 + K_{n=1}^{\infty}\frac{c_{n}}{1}, \notag
\end{align*}
}
where, for $n \geq 3$,
{\allowdisplaybreaks
\begin{align*}
%&\phantom{a}\notag\\
c_{n} &= \frac{a_{2n-2}a_{2n-1}}{(a_{2n-2}+a_{2n-3}+1)(a_{2n}+a_{2n-1}+1)}.\\
%&\phantom{a}\notag
\end{align*}
}
By  \eqref{con2},  the fact that each of the $a_{i}(q)$'s has
the same leading coefficient and the fact that if $|q|>1$ then
$\lim_{i \to \infty} 1/a_{i} = 0$, it follows that
{\allowdisplaybreaks
\begin{align*}
\lim_{n \to \infty}c_{n} &=
\lim_{n \to \infty}\frac{1}{(1+a_{2n-3}/a_{2n-2}+1/a_{2n-2})
(a_{2n}/a_{2n-1}+1+1/a_{2n-1})}\\
&=\frac{1}{(1+q^{m})(1+q^{-m})}:=c.\notag
%&\phantom{a}\notag
\end{align*}
}
Hence $G_{e}(q)$ is limit $1$-periodic.
Note that the value of $c$ depends on $q$.

Let the fixed points of $t(w)=c/(1+w)$ be denoted  $x$ and $y$.
From the remarks following \eqref{parelox}, it is clear that
 $t(w)$ is parabolic
only in the case $-1/((1+q^{m})(1+q^{-m})) = -1/4$. The only
solution to this equation is $q^{m}=1$, so that $t(w)$ is not
parabolic for any point outside the unit circle.

Similarly, $t(w)$ is elliptic only when $ -
1/((1+q^{m})(1+q^{-m})=-1/4-v$, for some real positive number $v$.
The solutions to this equation satisfy $q^{m} = (i  + \sqrt{v})/(
i  - \sqrt{v})$ or $q^{m} = ( i  - \sqrt{v})/( i + \sqrt{v})$.
However, it is easily seen that these are points on the unit
circle.

In all other cases $t(w)$ is loxodromic and $G_{e}(q)$ converges
in $\hat{\mathbb{C}}$. This proves the result for $G_{e}(q)$.

Similarly, by \eqref{op},
 the odd part of $G(q)$ is given by
{\allowdisplaybreaks
\begin{align*}
%&\phantom{a}\notag\\
&G_{o}(q)
:=\frac{1+a_{1}}{1}
-\frac{a_{1}a_{2}}{a_{3}+a_{2}+1}
\-
 \frac{a_{3}a_{4}}{a_{5}+a_{4}+1}
\-
 \frac{a_{5}a_{6}}{a_{7}+a_{6}+1}
\-
\cds . \notag\\
%&\phantom{a}\notag
\end{align*}
}
The proof in this case is virtually identical.
\end{proof}
As an  application of the above theorem, we have the following
example.
\begin{example} \label{eo1}
If $|q|>1$, then the odd and even parts of
{\allowdisplaybreaks
\begin{multline*}
G(q)=
1 + \frac{6q}{1}
\+
 \frac{3q^{2}+7q}{1}
\+
 \frac{3q^{3}+5q^{2}}{1}
\+
 \frac{q^{4}+7q^{3}+3 q+2}{1}
 \+\\
\frac{q^{5}+3 q^{4}+2q^{3}}{1}
\+
 \frac{q^{6}+2q^{5}+7 q^{3}}{1}
\+
 \frac{q^{7}+7q^{5}}{1}
\+
 \frac{q^{8}+7q^{6}+3q^{3}+2q  }{1}
\+\,\cds \\
\cds
\+
\frac{q^{4n+1}+3 q^{3n+1}+2q^{2n+1}}{1}
\+
 \frac{q^{4n+2}+2q^{3n+2}+7 q^{2n+1}}{1} \\
%\phantom{asdsfsfsf}
\+
 \frac{q^{4n+3}+5q^{3n+2}+2q^{2n+3}}{1}
\+
 \frac{q^{4n+4}+7q^{3n+3}+3q^{2n+1}+2q^{n}  }{1}
\+\,\cds
\end{multline*}
}
converge.
\end{example}
\begin{proof}
Let $k=4$ and
{\allowdisplaybreaks
\begin{align*}
f_{1}(x)&=qx^{4}+3qx^3+2qx^{2},\\
f_{2}(x)&=q^{2}x^{4}+2q^{2}x^3+7qx^{2},\\
f_{3}(x)&=q^{3}x^{4}+5q^{2}x^3+2q^{3}x^{2},\\
f_{4}(x)&=q^{4}x^{4}+7q^{3}x^3+3qx^{2} + 2x.
\end{align*}
}
Then, for $n \geq 0$ and $1 \leq j \leq 4$,
\begin{equation*}
a_{4n+j}(q) = f_{j}(q^{n}).
\end{equation*}
Thus  \eqref{con4} is satisfied. It is clear that
 \eqref{con2} is satisfied with
$M=1$ and each $a_{n}(q)$ has the same leading coefficient,
namely, 1.
\end{proof}
Remark: It is clear form Theorem \ref{T4}
 that if $k=1$ and $f_{i}(x)$ is \emph{any} polynomial with coefficients in
$\mathbb{Z}[q]$, then the odd and even parts of
$1+K_{n=0}^{\infty}f_{1}(q^{n})/1$ converge everywhere outside the unit circle
to values in $\hat{\mathbb{C}}$,
 since all the conditions of the theorem are satisfied automatically, at least for
a tail of the continued fraction.

\vspace{20pt}

We also consider continued fractions of the form
\begin{align*}
G(q):=b_{0}(q) +&
K_{n=1}^{\infty}\frac{a_{n}(q)}{b_{n}(q)}\\
:=g_{0}(q^{0})+&
\frac{f_{1}(q^{0})}{g_{1}(q^{0})}
 \+\,\cds \+
\frac{f_{k-1}(q^{0})}{g_{k-1}(q^{0})}
\+
\frac{f_{k}(q^{0})}{g_{0}(q^{1})} \notag \\
\+
&\frac{f_{1}(q^{1})}{g_{1}(q^{1})}
 \+\,\cds \+
\frac{f_{k-1}(q^{1})}{g_{k-1}(q^{1})}
\+
\frac{f_{k}(q^{1})}{g_{0}(q^{2})} \+\notag \\
  \cds \+
\frac{f_{k}(q^{n-1})}{g_{0}(q^{n})}
\+
&\frac{f_{1}(q^{n})}{g_{1}(q^{n})}
 \+\,\cds \+
\frac{f_{k-1}(q^{n})}{g_{k-1}(q^{n})} \+
\frac{f_{k}(q^{n})}{g_{0}(q^{n+1})} \+\cds \notag
\end{align*}
where $f_{s}(x), g_{s-1}(x) \in \mathbb{Z}[q][x]$, for $1 \leq s
\leq k$. Thus, for $n\geq 0$ and $1 \leq s \leq k$,
\begin{align}\label{con4ab}
&a_{nk+s}(q)=f_{s}(q^{n}),& &b_{nk+s-1}(q)=g_{s-1}(q^{n}).&
\end{align}
An example of a continued fraction of this type is the
G\"{o}llnitz-Gordon continued fraction (with $k=1$).

We suppose that degree $(a_{1}(q))=r_{1}$,
degree $(b_{0}(q))= r_{2}$, and that,
 for $i \geq 1$,
%{\allowdisplaybreaks
\begin{align}\label{con5}
%&\phantom{as} \notag \\
 \text{degree}(a_{i+1}(q))&= \text{degree}(a_{i}(q))+ a, \\
\text{degree}(b_{i}(q))&= \text{degree}(b_{i-1}(q))+ b, \notag
\end{align}
%}
where $a$ and $b$ are fixed positive integers and $r_{1}$ and
$r_{2}$ are non-negative integers. Condition \ref{con5} means
that, for $n \geq 1$,
\begin{align}\label{degeq}
 &\text{degree}(a_{n}(q))=(n-1)a + r_{1},&
& \text{degree}(b_{n}(q))=n\,b + r_{2}.
\end{align}
We also supposed that
each $a_{n}(q)$ has the same leading coefficient  $L_{a}$ and that each $b_{n}(q)$
has the same leading coefficient $L_{b}$.

For such continued fractions we have the following theorem.
%{\allowdisplaybreaks
\begin{theorem}\label{T:p2}
Suppose $G(q) = b_{o} + K_{n=1}^{\infty}a_{n}(q)/b_{n}(q)$
is such that the $a_{n}:= a_{n}(q)$ and the
 $b_{n}:= b_{n}(q)$ satisfy \eqref{con4ab} and \eqref{con5}.
 Suppose further that
 each $a_{n}(q)$ has the same leading
coefficient $L_{a}$ and that each $b_{n}(q)$ has the same leading
coefficient  $L_{b}$. If $2b>a$ then $G(q)$ converges everywhere
outside the unit circle. If $2b =a$,
 then  $G(q)$  converges outside the unit circle
 to values in $\hat{\mathbb{C}}$,
except possibly at points $q$ satisfying
$L_{b}^2/L_{a}q^{b-r_{1}+2r_{2}} \in
\left[-4, 0\right) $.
If $2b<a$, then the odd and even parts of
$G(q)$ converge everywhere outside the unit circle.
\end{theorem}
%}
\begin{proof}.
Let $|q|>1$. We first consider the case $2b>a$.
By a simple transformation, we have that
{\allowdisplaybreaks
\begin{equation*}
b_{0}+ K_{n=1}^{\infty}\frac{a_{n}}{b_{n}} \approx
b_{0}+\frac{a_{1}/b_{1}}{1}
\+
K_{n=2}^{\infty}\frac{a_{n}/(b_{n}b_{n-1})}{1}.
\end{equation*}
}
Since $2b>a$, $a_{n}/(b_{n}b_{n-1}) \to 0$ as
$n \to \infty$, and  $G(q)$ converges to a value in $\hat{\mathbb{C}}$,
by Worpitzky's  theorem.

Suppose $2b=a$. Then, by
 \eqref{con5}, \eqref{degeq} and the fact that each $a_{n}(q)$ has the same leading
coefficient  $L_{a}$ and that each $b_{n}(q)$
has the same leading coefficient $L_{b}$,
{\allowdisplaybreaks
\begin{align*}
\lim_{n \to \infty} \frac{a_{n}}{b_{n}b_{n-1}} =
\frac{L_{a}}{L_{b}^{2}q^{b - {r_1} + 2\,{r_2}}}:=c.
\end{align*}
}
Note once again that the value of $c$ depends on $q$.
Once again, by the remarks following \eqref{parelox},
the linear fractional
transformation $t(w) = c/(1+w)$ is parabolic only in the case
$L_{a}/(L_{b}^{2}q^{b - {r_1} + 2\,{r_2}})=-1/4$ or
$q^{b-r_{1}+2r_{2}} = -4L_{a}/L_{b}^{2}$.

Similarly,  $t(w)$ is elliptic only when
$q^{-b + {r_1} - 2\,{r_2}}\,           {L_a}/{{L_b}}^2
\in \left(-\infty,-1/4\right)$,
or
\begin{align*}
q^{b-r_{1}+2r_{2}}&=
\frac{-4\,{L_a}}{\left( 1 + 4\,v \right) \,{{L_b}}^2},
\end{align*}
for some real positive number $v$.
 In other words,
  $t(w)$ is elliptic (for $|q|>1$) only when
$q^{b-r_{1}+2r_{2}}$ lies either in the open interval
$(-4 L_{a}/L_{b}^{2},0)$ or $(0,-4 L_{a}/L_{b}^{2})$,
depending on the sign of $L_{a}$. In all other cases,
$t(w)$ is loxodromic, and $G(q)$ converges.

Suppose $2\,b < a$. From  \eqref{ep} it is clear that the even part of
$G(q)=b_{0}+K_{n=1}^{\infty}a_{n}/b_{n}$
can be transformed into the form
$b_{0} + K_{n=1}^{\infty}c_{n}/1$, where, for $n \geq 3$,
%{\allowdisplaybreaks
\begin{align*}
c_{n} &= \frac{-a_{2n-2}a_{2n-1}
\displaystyle{
\frac{b_{2n}}{b_{2n-2}}}}
{\left(a_{2n-2}+b_{2n-3}b_{2n-2}+a_{2n-3}
\displaystyle{
\frac{b_{2n-2}}{b_{2n-4}}}\right)
\left(a_{2n}+b_{2n-1}b_{2n}+a_{2n-1}
\displaystyle{
\frac{b_{2n}}{b_{2n-2}}}\right)}\\
&=\frac{
\displaystyle{
\frac{-a_{2n-1}b_{2n}}{a_{2n}b_{2n-2}}}}
{\left(1+\displaystyle{
\frac{b_{2n-3}b_{2n-2}}{a_{2n-2}}}
+
\displaystyle{
\frac{a_{2n-3}b_{2n-2}}{a_{2n-2}b_{2n-4}}}\right)
\left(1+\displaystyle{
\frac{b_{2n-1}b_{2n}}{a_{2n}}}
+
\displaystyle{
\frac{a_{2n-1}b_{2n}}{a_{2n}b_{2n-2}}}\right)}. \notag
\end{align*}
%}
Once again using \eqref{con5},  \eqref{degeq}
and the fact that each $a_{n}(q)$ has the same leading
coefficient  $L_{a}$ and that each $b_{n}(q)$
has the same leading coefficient $L_{b}$, we have that
{\allowdisplaybreaks
\begin{align*}
\lim_{n \to \infty} c_{n}
= -\frac{q^{2b-a}}
{\left(1+ q^{2b-a}\right)^{2}}
:=c.
\end{align*}
}
The linear fractional
transformation $t(w) = c/(1+w)$ is parabolic only in the case
$ -q^{2b-a}/ (1+ q^{2b-a})^{2}=-1/4$ or $q^{2b-a}=1$, and thus $|q|=1$.
It is elliptic only when
$ -q^{2b-a}/ (1+ q^{2b-a})^{2} \in (-\infty,-1/4)$, and a simple
argument shows that this implies that
 $|q^{2b-a}|=1$, and again $|q|=1$.

In all other cases $t(w)$ is loxodromic, and the even part of $G(q)$
converges by Theorem \ref{lox}.

The proof for the odd part of $G(q)$ is very similar and is
omitted.
\end{proof}
Remarks: (1) Worpitzky's Theorem once again gives weaker results.
In the example below, for example,  Worpitzky's Theorem gives that
$G(q)$ converges for  $|q| > 4$, in contrast to the result from
our theorem, which says that $G(q)$ converges everywhere outside the unit
circle, except possibly for $q \in [-4,-1)$.

(2) In some cases the result is the best possible. Numerical evidence
suggests that the continued fraction below converges nowhere in the
interval $(-4,-1)$.

As an application of Theorem \ref{T:p2}, we have the following example.
\begin{example}\label{eo2}
If $|q|>1$, then
{\allowdisplaybreaks
\begin{multline*}
G(q)=
q+2 + \frac{6q^{2}}{q^{2}+2}
\+
 \frac{3q^{4}+7q^{2}}{q^{3}+2}
\+
 \frac{3q^{6}+5q^{4}}{q^{4}+2}
\+
 \frac{q^{8}+7q^{6}+3 q^{2}+2}{q^{5}+q+1}
 \+\\
\frac{q^{10}+3 q^{8}+2q^{6}}{q^{6}+q^{2}+1}
\+
 \frac{q^{12}+2q^{10}+7 q^{6}}{q^{7}+q^{2}+1}
\+
 \frac{q^{14}+7q^{10}}{q^{8}+q^{3}}
\+
 \frac{q^{16}+7q^{12}+3q^{6}+2q^{2}  }{q^{9}+q^{2}+1}\\
\+
\cds
\+
\frac{q^{8n+2}+3 q^{6n+2}+2q^{4n+2}}{q^{4n+2}+q^{2n}+1}
\+
 \frac{q^{8n+4}+2q^{6n+4}+7 q^{4n+2}}{q^{4n+3}+q^{2n}+1} \\
%\phantom{asdsfsfsf}
\+
 \frac{q^{8n+6}+5q^{6n+4}+2q^{4n+6}}{q^{4n+4}+q^{3n}+1}
\+
 \frac{q^{8n+8}+7q^{6n+6}+3q^{4n+2}+2q^{2n}  }{q^{4(n+1)+1}+q^{n+1}+1}
\+\,\cds
\end{multline*}
}
converges, except possibly for  $q \in [-4,-1)$.
\end{example}

\begin{proof}
Let $k=4$ and
\begin{align*}
f_{1}(x)&=q^{2}x^{8}+3q^{2}x^6+2q^{2}x^{4},\\
f_{2}(x)&=q^{4}x^{8}+2q^{4}x^6+7q^{2}x^{4},\\
f_{3}(x)&=q^{6}x^{8}+5q^{4}x^6+2q^{6}x^{4},\\
f_{4}(x)&=q^{8}x^{8}+7q^{6}x^6+3q^{2}x^{4} + x^{2}\\
g_{0}(x)&=qx^{4}+x+1,\\
g_{1}(x)&=q^{2}x^{4}+x^{2}+1,\\
g_{2}(x)&=q^{3}x^{4}+x^{2}+1,\\
g_{3}(x)&=q^{4}x^{4}+x^{3}+1.
\end{align*}
Then, for $n \geq 0$ and $1 \leq j \leq 4$,
\begin{align*}
a_{4n+j}(q) &= f_{j}(q^{n}),\\
b_{4n+j-1}(q) &= g_{j-1}(q^{n}).
\end{align*}
The other requirements of the theorem are satisfied, with
$L_{a}=L_{b}=1$, $a=2$, $b=1$, $r_{1}=2$ and $r_{2}=1$.
Therefore $b-r_{1}+2r_{2}=1$, $L_{a}/L_{b}^{2}=1$ and
 $G(q)$ converges outside the unit circle,
except possibly for  $q \in [-4,-1)$.
\end{proof}

{\allowdisplaybreaks

}


\begin{thebibliography}{99}

\bibitem{ABJL92}
Andrews, G. E.; Berndt, Bruce C.; Jacobsen, Lisa; Lamphere, Robert L.
\emph{The continued fractions
found in the unorganized portions of Ramanujan's notebooks}.
Mem. Amer. Math. Soc. \textbf{99} (1992), no. 477, vi+71pp



\bibitem{BML01}
Bowman, D; Mc Laughlin, J
\emph{On the Divergence of the Rogers-Ramanujan Continued Fraction on
the Unit Circle. }
To appear  in the Transactions of the American Mathematical Society.



\bibitem{BML03}
Bowman, D; Mc Laughlin, J
\emph{On the Divergence  of $q$-Continued Fraction on
the Unit Circle. }
Submitted for publication.



\bibitem{BML02}
Bowman, D; Mc Laughlin, J
\emph{On the Divergence in the General Sense of $q$-Continued Fraction on
the Unit Circle. }
Submitted for publication.





\bibitem{J86}
Jacobsen, Lisa
\emph{General convergence of continued fractions}.
Trans. Amer. Math. Soc. \textbf{294} (1986), no. 2, 477--485.

\bibitem{JT80}   William B. Jones and W.J. Thron,
\emph{Continued Fractions Analytic Theory and Applications},
Addison-Wesley, London-Amsterdam-Ontario-Sydney-Tokyo,1980.


\bibitem{LW92}
Lorentzen, Lisa; Waadeland, Haakon
\emph{Continued fractions with applications}. Studies in
Computational Mathematics, 3. North-Holland Publishing Co.,
Amsterdam, 1992, pp 35--36.



\bibitem{Z91}
Zhang, Liang Cheng
\emph{$q$-difference equations and Ramanujan-Selberg continued fractions.}
Acta Arith. \textbf{57} (1991), no. 4, 307--355.
\end{thebibliography}
\end{document}